\documentclass[review]{elsarticle}

\usepackage{lineno,hyperref}
\modulolinenumbers[5]

\usepackage[utf8]{inputenc}
\usepackage[english]{babel}
\usepackage{mathtools}
\usepackage{amsthm}
\usepackage{amssymb}
\newtheorem{theorem}{Theorem}
\newtheorem{definition}{Definition}
\makeatletter
\def\ps@pprintTitle{%
\let\@oddhead\@empty
\let\@evenhead\@empty
\def\@oddfoot{\centerline{\thepage}}%
\let\@evenfoot\@oddfoot}
\makeatother

\begin{document}

\begin{frontmatter}

\title{On a generalized Birnbaum Saunders Distribution}

\author{ Beenu Thomas and Chacko V. M.}
\address{Department of Statistics}


\author[mysecondaryaddress]{St. Thomas' College (Autonomous)Thrissur, Kerala, India}

\begin{abstract}
In this paper, a  generalization for the Birnbaum Saunders
distribution, which has been applied to the modelling of fatigue
failure times and reliability studies, is considered. The maximum
likelihood estimators and statistical inference for the distribution
parameters are presented. Corresponding bivariate and multivariate
distributions are proposed. The proposed distribution is applied to
model real data sets.
\end{abstract}


\end{frontmatter}


\section{INTRODUCTION}

\par Normal distribution is the commonly used statistical distribution. Several distributions have been evolved by making some transformations on
Normal distribution. Two-parameter Birnbaum-Saunders(BS) distribution is one such distribution introduced by Birnbaum and Saunders\cite{BS(1969a)},
 which have been developed by employing a monotone transformation on the standard normal distribution.
Bhattacharyya and Fries\cite{BF(1982)} established that a BS distribution can be obtained as an approximation of an inverse Gaussian (IG) distribution.
Desmond\cite{D(1986)} examined an interesting feature that a BS distribution can be viewed as an equal mixture of an IG distribution and its reciprocal.
This property becomes helpful in deriving some properties of the BS distribution using properties of IG distribution.

\par Many different properties of BS distribution have been discussed by a number of authors. It has been observed that the probability density function
 of the BS distribution is unimodal. The shape of the hazard function (HF) plays an important role in lifetime data analysis.
 Kundu et al.\cite{kkb(2008)} and Bebbington et al.\cite{BLZ(2008)}
 proved that HF of BS distriution is an unimodal function. The maximum likelihood estimator(MLE)s of the shape and
 scale parameters based on a complete sample were discussed by Birnbaum and Saunders\cite{BS(1969b)}.
 The estimators of the unknown parameters in the case of a censored sample was first developed by Rieck\cite{R(1995)}.
 Ng et al.\cite{nkb(2003)} provided modified moment estimator(MME)s of the parameters in the case of a complete sample which are in explicit form.
 Li\cite{fl(2006)} suggested four different estimation techniques for both complete and censored samples.

\par Several other models associated with BS distribution and their properties have been discussed in literature.
For example, Rieck and Nedelman\cite{RN(1991)} considered the
log-linear model for the BS distribution. Desmond et
al.\cite{DRL(2008)} considered the BS regression model. Lemonte and
Cordeiro\cite{LC(2009)} considered the BS non-linear regression
model. Kundu et al.\cite{kbj(2010)} introduced the bivariate BS
distribution and studied some of its properties and characteristics.
The multivariate generalized BS distribution has been introduced, by
replacing the normal kernel by an elliptically symmetric kernel, by
Kundu et al.\cite{kbj(2010)}. A generalization of BS distribution is
done by Chacko et al.\cite{cmb(2015)}. Further research in inference
and its generalization to bivariate and multivariate framework is
still not addressed.

\par In this paper,  the generalization of Birnbaum Saunders  distribution is considered.
Section 2 reviewed univariate, bivariate and multivariate Birnbaum
Saunders distribution. Section 3  discussed the univariate
generalization of Birnbaum Saunders distribution. Section 4
discussed    the bivariate generalization of Birnbaum Saunders
distribution. Section 5 discussed   the multivariate generalization
of Birnbaum Saunders distribution. Application to real data set is
given in section 6. Conclusions are given at the last section.

\section{BIRNBAUM SAUNDERS DISTRIBUTION}
\subsection{\textbf{Univariate Birnbaum Saunders Distribution}}
A random variable following the BS distribution is defined through a
standard Normal random variable. Therefore the probability density
function(pdf) and cumulative distribution function(cdf) of the BS
model can be expressed in terms
of the standard Normal pdf and cdf.\\
The cdf of a two parameter BS random variable T for $\alpha>0, \beta>0$ can be written as
\begin{eqnarray}
F_T(t;\alpha,\beta)=\Phi \bigg[ \frac{1}{\alpha}\bigg(\frac{t}{\beta}\bigg)^{\frac{1}{2}}-\bigg(\frac{\beta}{t}\bigg)^{\frac{1}{2}}\bigg], t >0
\end{eqnarray}
\noindent where $\Phi(.)$ is the standard Normal cdf. The pdf of BS
distribution is
\begin{eqnarray}
  f_T(t;\alpha,\beta)=\frac{1}{2\sqrt(2\pi)\alpha\beta}\bigg[{\bigg(\frac{\beta}{t}\bigg)^{\frac{1}{2}}+\bigg(\frac{\beta}{t}\bigg)^{\frac{3}{2}}}\bigg].
  exp\bigg[-\frac{1}{2\alpha^2}\bigg(\frac{t}{\beta}+\frac{\beta}{t}-2\bigg)\bigg].
\end{eqnarray}
Here $\alpha > 0$ and $\beta >0$ are the shape and scale parameters respectively.

\subsection{\textbf{Bivariate Birnbaum Saunders Distribution}}
\par The bivariate Birnbaum-Saunders (BVBS) distribution was introduced by Kundu et.al\cite{kbj(2010)}.
The bivariate random vector $(T_1, T_2)$ is said to have a BVBS
distribution with parameters $\alpha_1, \beta_1, \alpha_2, \beta_2,
\rho$ if the cumulative distribution function of $(T_1, T_2)$ can be
expressed as
\begin{eqnarray}
  F(t_1, t_2) &=& \Phi_2\bigg[\frac{1}{\alpha_1}\bigg(\sqrt(\frac{t_1}{\beta_1})-\sqrt(\frac{\beta_1}{t_1})\bigg),\frac{1}{\alpha_2}\bigg(\sqrt(\frac{t_2}{\beta_2})-\sqrt(\frac{\beta_2}{t_2})\bigg)\bigg]
\end{eqnarray}
for $t_1 > 0, t_2 > 0, \alpha_1 > 0, \beta_1 > 0, \alpha_2, \beta_2
> 0 and -1 < \rho < 1$. Here $\Phi_2(u, v; \rho)$ is the cdf of
standard bivariate normal vector $(Z_1, Z_2)$ with correlation
coefficient $\rho$. The corresponding pdf is
\begin{multline*}
  f(t_1, t_2) = \frac{1}{8\pi \alpha_1\alpha_2\beta_1\beta_2\sqrt(1-\rho^2)}\bigg[(\frac{\beta_1}{t_1})^{\frac{1}{2}}+(\frac{\beta_1}{t_1})^{\frac{3}{2}}\bigg]
  \bigg[(\frac{\beta_2}{t_2})^{\frac{1}{2}}+(\frac{\beta_2}{t_2})^{\frac{3}{2}}\bigg]\\
  e^{-\frac{1}{2(1-\rho^2)}\bigg[\frac{1}{\alpha_1^2}\bigg(\sqrt(\frac{t_1}{\beta_1})-\sqrt(\frac{\beta_1}{t_1})\bigg)^2+\frac{1}{\alpha_2^2}\bigg(\sqrt(\frac{t_2}{\beta_2})-\sqrt(\frac{\beta_2}{t_2})\bigg)^2-\frac{2\rho}{\alpha_1\alpha_2}\bigg(\sqrt(\frac{t_1}{\beta_1})-\sqrt(\frac{\beta_1}{t_1})\bigg)\bigg(\sqrt(\frac{t_1}{\beta_1})-\sqrt(\frac{\beta_1}{t_1})\bigg)\bigg]}
\end{multline*}
for $t_1 > 0, t_2 > 0, \alpha_1 > 0, \beta_1 > 0, \alpha_2, \beta_2 > 0 and -1 < \rho < 1$.
\subsection{\textbf{Multivariate Birnbaum Saunders Distribution}}
Kundu et al.\cite{kbj(2013)} introduced the multivariate BS distribution.
Let $\underline{\alpha},\underline{\beta}\in \mathbb{R}^p$, where $\underline{\alpha}=(\alpha_1,\cdots,\alpha_p)^T$ and $\underline{\beta}=(\beta_1,\cdots,\beta_p)^T$, with $\alpha_i > 0, \beta_i > 0$ for $i= 1,\cdots,p$. Let $\mathbf{\Gamma}$ be a $p\times p$ positive-definite correlation matrix. Then, the random vector $\underline{T}=(T_1,\cdots,T_p)^T$ is said to have a p-variate BS distribution with parameters $(\underline{\alpha},\underline{\beta},\mathbf{\Gamma})$ if it has the joint CDF as
\begin{align*}
P(\underline{T}\leq\underline{t})&= P(T_1 \leq t_1,\cdots, T_p \leq t_p)\\
&= \Phi_p\bigg[\frac{1}{\alpha_1}\bigg(\sqrt(\frac{t_1}{\beta_1})-\sqrt(\frac{\beta_1}{t_1})\bigg),\cdots,\frac{1}{\alpha_p}\bigg(\sqrt(\frac{t_p}{\beta_p})-\sqrt(\frac{\beta_p}{t_p})\bigg);\mathbf{\Gamma}\bigg]
\end{align*}
for $t_1 > 0,\cdots, t_p > 0.$ Here, for $\underline{u}=(u_1,\cdots,
u_p)^T, \Phi_p(\underline{u};\mathbf{\Gamma})$ denotes the joint cdf
of a standard normal vector $\underline{Z}=(Z_1,\cdots, Z_p)^T$ with
correlation matrix $\mathbf{\Gamma}$. The joint pdf of
$\underline{T}=(T_1,\cdots,T_p)^T$ can be obtained from the above
equation as
\begin{align*}
f_{\underline{T}}(\underline{t}; \underline{\alpha}, \underline{\beta}, \mathbf{\Gamma})&= \phi_p\bigg(\frac{1}{\alpha_1}\bigg(\sqrt(\frac{t_1}{\beta_1})-\sqrt(\frac{\beta_1}{t_1})\bigg),\cdots,\frac{1}{\alpha_p}\bigg(\sqrt(\frac{t_p}{\beta_p})-\sqrt(\frac{\beta_p}{t_p})\bigg);\mathbf{\Gamma}\bigg)\\
&\times \prod_{i=1}^p\frac{1}{2\alpha_i\beta_i}\{\bigg(\frac{\beta_i}{t_i}\bigg)^{\frac{1}{2}}+\bigg(\frac{\beta_i}{t_i}\bigg)^{\frac{3}{2}}\},
\end{align*}
for $t_1 > 0,\cdots, t_p > 0;$ here, for $\underline{u}=(u_1,\cdots, u_p)^T,$
$$\phi_p(u_1,\cdots,u_p;\mathbf{\Gamma})=\frac{1}{(2\pi)^{\frac{p}{2}}|\mathbf{\Gamma}|^{\frac{1}{2}}}exp\{-\frac{1}{2}\underline{u}^T\mathbf{\Gamma}^{-1}\underline{u}\}$$

is the pdf of the standard Normal vector with correlation matrix $\mathbf{\Gamma}$.\\
Now consider the generalization of BS distribution.
\section{UNIVARIATE $\nu$-BIRNBAUM SAUNDERS DISTRIBUTION}
\par Following Chacko et al. \cite{cmb(2015)},   consider  the univariate $\nu-$BS
distribution.\\ Let
$\xi_\nu(t)=[(\frac{t}{\beta})^\nu-(\frac{t}{\beta})^{-\nu}]$
instead of
$\xi(t)=[(\frac{t}{\beta})^{\frac{1}{2}}-(\frac{t}{\beta})^{-{\frac{1}{2}}}]$
in the univariate BS distribution. The cdf of a univariate $\nu$-BS
random variable T can be written as
\begin{equation}\label{1}
F(t;\alpha,\beta,\nu)= \Phi(\frac{1}{\alpha}\xi_\nu(\frac{t}{\beta}));t > 0, \alpha,\beta>0, 0<\nu<1
\end{equation}
where $\Phi(.)$ is the standard normal cdf. The parameters $\alpha$
and $\beta$ in \ref{1}   are the shape and scale parameters,
respectively.
The parameter $\nu$ governs both scale and shape.\\
If the random variable T has the BS distribution function in
\ref{1}, then the corresponding pdf is
\begin{equation}\label{2}
f(t)=\frac{\nu}{\alpha\beta\sqrt(2\pi)}e^{-{\frac{1}{2\alpha^2}[(\frac{t}{\beta})^{2\nu}+(\frac{t}{\beta})^{-2\nu}-2][(\frac{t}{\beta})^{\nu-1}+
(\frac{t}{\beta})^{-(\nu+1)}]}}.
\end{equation}

The pdf is unimodal. Now the moments can be obtained as below.
 \subsection{\textbf{Moments}}
 If T is a $\nu$- BS distribution, denoted as $\nu- BS(\alpha,\beta)$, then the moments of this distribution are obtained by making transformation
 $$X=\frac{1}{2}[(\frac{T}{\beta})^\nu-(\frac{T}{\beta})^{-\nu}].$$
 $$X^2=\frac{1}{4}[(\frac{T}{\beta})^{2\nu}-(\frac{T}{\beta})^{-2\nu}-2]\Rightarrow$$
 $$4X^2+2=(\frac{T}{\beta})^{2\nu}-(\frac{T}{\beta})^{-2\nu}=(\frac{\beta}{T})^{2\nu}[1+(\frac{T}{\beta})^{4\nu}]\Rightarrow$$
 $$(\frac{T}{\beta})^{4\nu}-(4X^2+2)(\frac{T}{\beta})^{2\nu}+1=0.$$
Using this, we get the   moments.

\subsection{\textbf{Property of $\nu-$Birnbaum Saunders Distribution}}
\begin{theorem}
If T has a $\nu$-BS distribution with parameters $\alpha, \beta$ and $\nu$ then $T^{-1}$ also has a $\nu$-BS distribution with parameters $\alpha, \beta^{-1}$ and $\nu^{-1}$
\end{theorem}

\begin{proof}
Put $y=\frac{1}{T}\Rightarrow T=\frac{1}{y};
~~\frac{dT}{dy}=-\frac{1}{y^2}$, ~~~
$|\frac{dT}{dy}|=\frac{1}{y^2}$,
\begin{multline*}
  f_y=f_T|\frac{dT}{dy}|=\frac{\nu}{\alpha \beta \sqrt(2\pi)}\big[((\frac{1}{\beta y})^(\nu-1))+((\frac{1}{\beta y})^{-(\nu+1)})\big]e^{-{\frac{1}{2\alpha^2}\big[(\frac{1}{\beta y})^{(2\nu)}+(\beta y)^{(2\nu)}-2\big]}}\frac{1}{y^2} \\
  =\frac{\nu}{\alpha\beta\sqrt(2\pi)}\big[(\frac{1}{\beta y})^{(\nu-1)}+(\beta y)^{(\nu+1)}e^{-{\frac{1}{2\alpha^2}\big[(\frac{1}{\beta y})^{(2\nu)}+(\beta y)^{(2\nu)}-2\big]}}\frac{1}{y^2}\big]\\
  =\frac{\nu}{\alpha\beta\sqrt(2\pi)}\big[{\frac{1}{\beta^{\nu-1}}y^{\nu+1}}+{\beta^{\nu+1}y^{\nu-1}}\big]e^{-{\frac{1}{2\alpha^2}\big[(\frac{1}{\beta y})^{(2\nu)}+(\beta y)^{(2\nu)}-2\big]}}\\
\end{multline*}

\end{proof}

\subsection{\textbf{Estimation}}
The estimation of parameters can be
done by method of maximum likelihood.
\subsubsection{\textbf{Maximum Likelihood Estimates}}
\par Let $T_1$,
$T_2$,...,$T_n$ be the random sample of size n. Based on a random sample,
the MLEs of the unknown parameters can be obtained by maximising the log-likelihood function.
The likelihood function is
$$L=(\frac{\nu}{\alpha\beta\sqrt(2\pi)})^n e^{-{\frac{1}{2\alpha^2}\sum^n_{i=1}\big[(\frac{t_i}{\beta})^{2\nu}+(\frac{t_i}
{\beta})^{-2\nu}-2\big]}}\prod^n_{i=1}[(\frac{t_i}{\beta})^{\nu-1}+(\frac{t_i}{\beta})^{-(\nu+1)}].$$

The log-likelihood is
\begin{equation*}
  \log L=n\log \nu-n\log \alpha-n\log \beta-\frac{n}{2}\log(2\pi)-\frac{1}{2\alpha^2}\sum^n_{i=1}\bigg[(\frac{t_i}{\beta})^{2\nu}+
  (\frac{t_i}{\beta})^{-2\nu}-2\bigg]+\sum^n_{i=1}\log\bigg[(\frac{t_i}{\beta})^{\nu-1}+(\frac{t_i}{\beta})^{-(\nu+1)}\bigg].
\end{equation*}
Equating the partial derivative of log-likelihood function with
respect to parameters, to zero, we get,
\begin{align*}
  \frac{\partial\log L}{\partial\alpha}=0 &\Rightarrow \frac{-n}{\alpha}+\frac{1}{\alpha^3}\sum^n_{i=1}\bigg[\bigg(\frac{t_i}{\beta}\bigg)^{2\nu}+\bigg(\frac{\beta}{t_i}\bigg)^{2\nu}-2\bigg]=0 \\
  &\Rightarrow \frac{1}{\alpha^2}\sum^n_{i=1}\bigg[\bigg(\frac{t_i}{\beta}\bigg)^{2\nu}+\bigg(\frac{\beta}{t_i}\bigg)^{2\nu}-2\bigg]=n\\
  &\Rightarrow \frac{1}{n}\sum^n_{i=1}\bigg[\bigg(\frac{t_i}{\beta}\bigg)^{2\nu}+\bigg(\frac{\beta}{t_i}\bigg)^{2\nu}-2\bigg]=\alpha^2\\
 & \Rightarrow \alpha =
 \{\frac{1}{n}\sum^n_{i=1}\bigg[\bigg(\frac{t_i}{\beta}\bigg)^{2\nu}+\bigg(\frac{\beta}{t_i}\bigg)^{2\nu}-2\bigg]\}^{\frac{1}{2}}.
\end{align*}

\begin{equation*}
 \frac{\partial\log L}{\partial\beta}=0 \Rightarrow \frac{n}{\beta}+\frac{\nu}{\alpha^2}\sum^n_{i=1}\bigg[\frac{\beta^{2\nu-1}}{t_i^{2\nu}}-\frac{t_i^{2\nu}}{\beta^{2\nu+1}}\bigg]+(\nu+1)\sum^n_{i=1}\frac{\frac{\beta^\nu}{t_i^{\nu+1}}-\frac{t_i^{\nu-1}}{\beta^\nu}}{(\frac{t_i}{\beta})^{\nu-1}+(\frac{\beta}{t_i})^{\nu+1}}=0\\
\end{equation*}
and
\begin{equation*}
  \frac{\partial\log L}{\partial\nu}=0 \Rightarrow \frac{n}{\nu}-\frac{1}{\alpha^2}\sum^n_{i=1}\bigg[(\frac{t_i}{\beta})^{2\nu}\log(\frac{t_i}
  {\beta})+(\frac{\beta}{t_i})^{2\nu}\log(\frac{\beta}{t_i})\bigg]+\sum^n_{i=1}\frac{(\frac{t_i}{\beta})^{\nu-1}\log(\frac{t_i}{\beta})+
  (\frac{\beta}{t_i})^{\nu+1}\log(\frac{\beta}{t_i})}{(\frac{t_i}{\beta})^{\nu-1}+(\frac{\beta}{t_i})^{\nu+1}}=0.
\end{equation*}
The equations can be solved numerically.

\section{BIVARIATE $\nu$- BIRNBAUM SAUNDERS DISTRIBUTION}
\par The bivariate random vector $(T_1, T_2)$ is said to have a bivariate BS distribution with parameters
 $\alpha_1, \beta_1, \nu_1, \alpha_2, \beta_2, \nu_2,\rho$ if the joint cdf of $(T_1, T_2)$ can be expressed as
\begin{align*}
F(t_1,t_2)&= P(T_1\leq t_1, T_2 \leq t_2)\\
&= \Phi_2
\bigg[\frac{1}{\alpha_1}((\frac{t_1}{\beta_1})^{\nu_1}-(\frac{\beta_1}{t_1})^{\nu_1}),\frac{1}{\alpha_2}((\frac{t_2}{\beta_2})^{\nu_2}-
(\frac{\beta_2}{t_2})^{\nu_2});\rho \bigg] ; t_1 > 0, t_2 > 0.
\end{align*}
Here $\alpha_1 > 0, \beta_1 > 0, \alpha_2 > 0, \beta_2 > 0, -1 <
\rho < 1$ and $\Phi_2(u, v; \rho)$ is cdf of standard bivariate
Normal vector $(z_1, z_2)$ with correlation coefficient $\rho$. The
corresponding joint pdf of $T_1$ and $T_2$ is given by
\begin{equation*}
  f_{T_1,T_2}(t_1,t_2)=\phi_2\bigg[\frac{1}{\alpha_1}((\frac{t_1}{\beta_1})^{\nu_1}-(\frac{\beta_1}{t_1})^{\nu_1}),
   \frac{1}{\alpha_2}((\frac{t_2}{\beta_2})^{\nu_2}-(\frac{\beta_2}{t_2})^{\nu_2});\rho\bigg]\frac{dT_1}{dt}\frac{dT_2}{dt}
\end{equation*}
where $\phi_2(u, v; \rho)$ denotes the joint pdf of $z_1$ and $z_2$
given by
$$\phi_2(u, v; \rho)=\frac{1}{2 \pi \sqrt(1-\rho^2)}e^{-\frac{1}{2(1-\rho^2)}(u^2+ v^2 - 2\rho uv)}.$$
\begin{align*}
\frac{dT_1}{dt_1}&=\frac{d}{dt_1}\bigg[\frac{1}{\alpha_1}((\frac{t_1}{\beta_1})^{\nu_1}-(\frac{\beta_1}{t_1})^{\nu_1})\bigg]\\
&= \frac{1}{\alpha_1}\frac{d}{dt_1}\bigg[(\frac{t_1}{\beta_1})^{\nu_1}-(\frac{\beta_1}{t_1})^{\nu_1}\bigg]\\
&= \frac{\nu_1}{\alpha_1}\bigg[\frac{t_1^{\nu_1-1}}{\beta_1^{\nu_1}}+\frac{\beta^\nu}{t_1^{\nu_1+1}}\bigg]\\
&=
\frac{\nu_1}{\alpha_1\beta_1}\bigg[(\frac{t_1}{\beta_1})^{\nu_1-1}+(\frac{\beta_1}{t_1})^{\nu_1+1}\bigg].
\end{align*}
Now
\begin{align*}
\frac{dT_2}{dt}&=\frac{d}{dt_2}\bigg[\frac{1}{\alpha_2}((\frac{t_2}{\beta_2})^{\nu_2}- (\frac{\beta_2}{t_2})^{\nu_2})\bigg]\\
&= \frac{1}{\alpha_2}\frac{d}{dt_2}\bigg[(\frac{t_2}{\beta_2})^{\nu_2}- (\frac{\beta_2}{t_2})^{\nu_2}\bigg]\\
&=
\frac{\nu_2}{\alpha_2\beta_2}\bigg[(\frac{t_2}{\beta_2})^{\nu_2-1}-
(\frac{\beta_2}{t_2})^{\nu_2+1}\bigg].
\end{align*}
\begin{multline*}
f_{T_1,T_2}(t_1,t_2)= \frac{\nu_1\nu_2}{2\pi\alpha_1\alpha_2\beta_1\beta_2\sqrt(1-\rho^2)}\bigg[(\frac{t_1}{\beta_1})^{\nu_1-1}+(\frac{\beta_1}{t_1})^{\nu_1+1}\bigg]
\bigg[(\frac{t_2}{\beta_2})^{\nu_2-1}+(\frac{\beta_2}{t_2})^{\nu_2+1}\bigg]\\ exp\{\frac{-1}{2(1-\rho^2)}\bigg[{\frac{1}{\alpha_1^2}((\frac{t_1}{\beta_1})^\nu_1-(\frac{\beta_1}{t_1})^\nu_1)^2}+{\frac{1}{\alpha_2^2}((\frac{t_2}{\beta_2})^\nu_2-(\frac{\beta_2}{t_2})^\nu_2)^2}-{\frac{2\rho}{\alpha_1\alpha_2}\bigg[(\frac{t_1}{\beta_1})^\nu_1-(\frac{\beta_1}{t_1})^\nu_1\bigg]\bigg[(\frac{t_2}{\beta_2})^\nu_2-(\frac{\beta_2}{t_2})^\nu_2\bigg]}\bigg]\}
\end{multline*}

\begin{theorem}
If $(T_1, T_2) \sim$ BV $\nu$-BS
$(\alpha_1,\beta_1,\nu_1,\alpha_2,\beta_2,\nu_2,\rho)$ then $T_i
\sim \nu- BS(\alpha_i,\beta_i,\nu_i)$.
\end{theorem}

\begin{proof}
\begin{align*}
f_{T_1}(t_1,t_2)&=\int_{t_2}f_{T_1,T_2}(t_1, t_2)dt_2\\
&=\frac{\nu_1\nu_2}{\alpha_1\alpha_2\beta_1\beta_2 2\pi\sqrt(1-\rho^2)}\bigg[\bigg(\frac{t_1}{\beta_1}\bigg)^{\nu_1-1}+\bigg(\frac{\gamma_1}{t_1}\bigg)^{\nu_1+1}\bigg]\int_0^\infty\bigg[\bigg(\frac{t_2}{\beta_2}\bigg)^{\nu_2-1}+\bigg(\frac{\beta_2}{t_2}\bigg)^{\nu_2+1}\bigg]\\
&e^{\frac{-1}{2(1-\rho^2)}\bigg[\frac{1}{\alpha_1^2}\bigg(\bigg(\frac{t_1}{\beta_1}\bigg)^{\nu_1}-
\bigg(\frac{\beta_1}{t_1}\bigg)^{\nu_1}\bigg)^2+\frac{1}{\alpha_2^2}\bigg(\bigg(\frac{t_2}{\beta_2}\bigg)^{\nu_2}-
\bigg(\frac{\beta_2}{t_2}\bigg)^{\nu_2}\bigg)^2-\frac{2\rho}{\alpha_1\alpha_2}\bigg[\bigg(\frac{t_1}{\beta_1}\bigg)^{\nu_1}-
\bigg(\frac{\beta_1}{t_1}\bigg)^{\nu_1}\bigg]\bigg[\bigg(\frac{t_2}{\beta_2}\bigg)^{\nu_2}-\bigg(\frac{\beta_2}{t_2}\bigg)^{\nu_2}\bigg]\bigg]}dt_2.
\end{align*}
\begin{align*}
Put~~ & u=(\frac{t_2}{\beta_2})^{\nu_2}-(\frac{\beta_2}{t_2})^{\nu_2}\\
&du=\frac{\nu_2}{\beta_2}\bigg[(\frac{t_2}{\beta_2})^{\nu_2-1}+(\frac{\beta_2}{t_2})^{\nu_2+1}\bigg]dt_2.
\end{align*}
\begin{eqnarray*}
  f(t_1) =\frac{\nu_1}{\alpha_1\beta_1\sqrt(2\pi)}\bigg[\bigg(\frac{t_1}{\beta_1}\bigg)^{\nu_1-1}+\bigg(\frac{\beta_1}{t_1}\bigg)^{\nu_1+1}\bigg]
  e^{-\frac{1}{2\alpha_1^2}\bigg[\bigg(\frac{t_1}{\beta_1}\bigg)^{2\nu_1}+\bigg(\frac{\beta_1}{t_1}\bigg)^{2\nu_1}-2\bigg]}.
\end{eqnarray*}
\end{proof}

\begin{theorem}
If $(T_1, T_2) \sim\nu-BS(\alpha_1,\beta_1,\nu_1,\alpha_2,\beta_2,\nu_2,\rho)$ then
\begin{enumerate}
  \item $(T_1^{-1}, T_2^{-1})\sim\nu-BS(\alpha_1,\frac{1}{\beta_1},\nu_1,\alpha_2,\frac{1}{\beta_2},\nu_2,\rho)$
  \item $(T_1^{-1}, T_2)\sim\nu-BS(\alpha_1,\frac{1}{\beta_1},\nu_1,\alpha_2,\beta_2,\nu_2,\rho)$
   \item $(T_1,
   T_2^{-1})\sim\nu-BS(\alpha_1,\beta_1,\nu_1,\alpha_2,\frac{1}{\beta_2},\nu_2,\rho)$
   \end{enumerate}
\end{theorem}
\subsubsection{\textbf{Estimation}}
\par Based on a bivariate random sample $\{(t_{1i}, t_{2i}), i = 1, 2, ..., n\}$ from the $\nu-BVBS(\alpha_1, \beta_1, \nu_1, \alpha_2, \beta_2, \nu_2, \rho)$
distribution, the MLEs of the unknown parameters can be obtained by
maximizing the log likelihood function. If we denote
$\theta=(\alpha_1, \beta_1, \nu_1, \alpha_2, \beta_2, \nu_2, \rho)$
then the likelihood function is
\begin{multline*}
  L(\theta)= \bigg(\frac{\nu_1\nu_2}{\alpha_1\alpha_2\beta_1\beta_2 2\pi \sqrt(1-\rho^2)}\bigg)^n \Pi_{i=1}^n\bigg[\bigg(\frac{t_{1i}}{\beta_1}\bigg)^{\nu_1-1}+\bigg(\frac{\beta_1}{t_{1i}}\bigg)^{\nu_1+1}\bigg]\bigg[\bigg(\frac{t_{2i}}{\beta_2}\bigg)^{\nu_2-1}+\bigg(\frac{\beta_2}{t_{2i}}\bigg)^{\nu_2+1}\bigg]\\
  exp\{-\frac{1}{2(1-\rho^2)}\Sigma_{i=i}^n\bigg[\frac{1}{\alpha_1^2}\bigg((\frac{t_{1i}}{\beta_1})^{\nu_1}-(\frac{\beta_1}{t_{1i}})^{\nu_1}\bigg)^2+ \frac{1}{\alpha_2^2}\bigg((\frac{t_{2i}}{\beta_2})^{\nu_2}-(\frac{\beta_2}{t_{2i}})^{\nu_2}\bigg)^2-\\
  \frac{2\rho}{\alpha_1\alpha_2}\bigg(\bigg(\frac{t_{1i}}{\beta_1}\bigg)^{\nu_1}-
  \bigg(\frac{\beta_1}{t_{1i}}\bigg)\bigg)\bigg(\bigg(\frac{t_{2i}}{\beta_2}\bigg)^{\nu_2}-\bigg(\frac{\beta_2}{t_{2i}}\bigg)\bigg)
  \bigg]\}.
\end{multline*}
The log-likelihood function is
\begin{multline*}
  \log L(\theta)=n \log \nu_1+n\log \nu_2-n\log \alpha_1 -n\log \alpha_2-n\log \beta_1-n \log \beta_2-n\log {2\pi}-\frac{n}{2}\log(1-\rho^2)\\
  +\Sigma_{i=1}^n\log \bigg[\bigg(\frac{t_{1i}}{\beta_1}\bigg)^{\nu_1-1}+\bigg(\frac{\beta_1}{t_{1i}}\bigg)^{\nu_1+1}\bigg]
  +\Sigma_{i=1}^n\log\bigg[\bigg(\frac{t_{2i}}{\beta_2}\bigg)^{\nu_2-1}+\bigg(\frac{\beta_2}{t_{2i}}\bigg)^{\nu_2+1}\bigg]\\
  -\frac{1}{2(1-\rho^2)}\bigg[{\Sigma_{i=1}^n\frac{1}{\alpha_1^2}\bigg(\bigg(\frac{t_{1i}}{\beta_1}\bigg)^{\nu_1}-\bigg(\frac{\beta_1}{t_{1i}}\bigg)^{\nu_1}\bigg)^2}+\Sigma_{i=1}^n\frac{1}{\alpha_2^2\bigg(\bigg(\frac{t_{2i}}{\beta_2}\bigg)^{\nu_2}-\bigg(\frac{\beta_2}{t_{2i}}\bigg)^{\nu_2}\bigg)^2}\\
  -\frac{2\rho}{\alpha_1\alpha_2}\Sigma_{i=1}^n\bigg[\bigg(\frac{t_{1i}}{\beta_1}\bigg)^{\nu_1}-
  \bigg(\frac{\beta_1}{t_{1i}}\bigg)^{\nu_1}\bigg]\bigg[\bigg(\frac{t_{2i}}{\beta_2}\bigg)^{\nu_2}-\bigg(\frac{\beta_2}{t_{2i}}\bigg)^{\nu_2}\bigg]\bigg].
\end{multline*}

\section{MULTIVARIATE $\nu$- BIRNBAUM SAUNDERS DISTRIBUTION }
\par Along the same lines as the univariate and bivariate $\nu-$BS distribution, the multivariate $\nu-$BS distribution can be defined.

Then the multivariate $\nu-$BS distribution is as follows:
\begin{definition}
Let $\underline{\alpha},\underline{\beta}\in \mathbb{R}^m$, where $\underline{\alpha}=(\alpha_1,\cdots, \alpha_m)^T$ and $\underline{\beta}=(\beta_1,\cdots,\beta_m)^T$, with $\alpha_1>0, \beta_i>0$ for $i=1,2,\cdots, m$. Let $\mathbf{\Gamma}$ be a $m\times m$ positive definite correlation matrix. Then, the random vector $\underline{T}=(T_1,\cdots, T_m)^T$ is said to have a m-variate BS distribution with parameters $(\underline{\alpha},\underline{\beta},\mathbf{\Gamma},\nu)$ if it has the joint CDF as
\begin{align*}
P(\underline{T}\leq\underline{t})&= P(T_1\leq t_1,\cdots, T_m \leq t_m)\\
&=\Phi_m\bigg[\frac{1}{\alpha_1}\bigg(\bigg(\frac{t_1}{\beta_1}\bigg)^\nu-\bigg(\frac{\beta_1}{t_1}\bigg)^\nu\bigg),\cdots,\frac{1}{\alpha_m}\bigg(\bigg(\frac{t_m}{\beta_m}\bigg)^\nu-\bigg(\frac{\beta_m}{t_m}\bigg)^\nu\bigg);\Gamma\bigg]
\end{align*}
for $t_1 > 0,\cdots, t_m > 0$ and $0 < \nu <1$. Here, for
$\underline{u}=(u_1,\cdots, u_m)^T,
\Phi_m(\underline{u};\mathbf{\Gamma})$ denotes the joint cdf of a
standard Normal vector $\underline{Z}=(Z_1,\cdots, Z_m)^T$ with
correlation matrix $\mathbf{\Gamma}$.
\end{definition}
\subsection{\textbf{Applications}}
Birnbaum and Saunders (1958) obtained fatigue life real data corresponding to cycles $(\times 10^{-3})$ until failure of aluminum specimens of type 6061-T6, see Table 1. These specimens were cut parallel to the direction of rolling and oscillating at 18 cycles per seconds. They were exposed to a pressure with maximum stress 31,000 pounds per square inch (psi) for $n= 101$ specimens for each level of stress. All specimens were tested until failure.\\

\begin{table}[h!]
\caption{Fatigue lifetime data  }
\begin{tabular}{|c|c|c|c|c|c|c|c|c|c|c|c|c|c|c|}
  \hline
  70 & 90 & 96 & 97 & 99 & 100 & 103 & 104 & 104 & 105 & 107 & 108 & 108 & 108 & 109 \\
  109 & 112 & 112 & 113 & 114 & 114 & 114 & 116 & 119 & 120 & 120 & 120 & 121 & 121 & 123 \\
  124 & 124 & 124 & 124 & 124 & 128 & 128 & 129 & 129 & 130 & 130 & 130 & 131 & 131 & 131 \\
  131 & 131 & 132 & 132 & 132 & 133 & 134 & 134 & 134 & 134 & 134 & 136 & 136 & 137 & 138 \\
  138 & 138 & 139 & 139 & 141 & 141 & 142 & 142 & 142 & 142 & 142 & 142 & 144 & 144 & 145 \\
  146 & 148 & 148 & 149 & 151 & 151 & 152 & 155 & 156 & 157 & 157 & 157 & 157 & 158 & 159 \\
  162 & 163 & 163 & 164 & 166 & 166 & 168 & 170 & 174 & 196 & 212 &  &  &  &  \\
  \hline
\end{tabular}
\end{table}
For this data set the point estimates of $\alpha, \beta$ and $\nu$ obtained by the method of maximum likelihood are given in the table 2.\\
\begin{center}\begin{table}[h!]
\caption{Point estimates of $\alpha, \beta$ and $\nu$}
\begin{tabular}{|c|c|c|}
  \hline
  $\alpha$ & $\beta$ & $\nu$ \\
  \hline
  1.509180e+00 & 1.179090e-06 & 6.198936e+00 \\

  \hline
\end{tabular}
\end{table}
\end{center}

Here Kolmogrov-Smirnov test statistic is 0.9703 and p-value is
0.3088. Since p-value is greater than significance level we can
conclude that univariate $\nu-$Birnbaum-Saunders distribution is a
good fit for given data.

\subsection{\textbf{Conclusion}}
\par This paper discussed a generalization of BS distribution.This three parameter distribution is more plausible model for the distribution of fatigue
failure. There is flexibility in selection of models. Peakedness
depends on the value of $\nu$. It is a good model for distributions
with smaller variances. Moreover the shape of density curve with
various skewness and kurtosis provide a well defined class of life
distributions useful in reliability and social sciences. Snedecors F
distribution has the property that reciprocal is also F, similar
property holds for $\nu-$Birnbaum Saunders distribution. But
computation of MLE is a complicated one. But numerical procedure is
applied for computation of MLE.


\begin{thebibliography}{}

\bibitem{BLZ(2008)}\textrm{Bebbington, M., Lai, C.D., Zitikis, R.(2008)}, \textrm{A proof of the shape of the Birnbaum-Saunders hazard rate function}, \emph{\textrm{The mathematical Scientist, Vol. 33, pp. 49-56}}.

\bibitem{BF(1982)}\textrm{Bhattacharyya, G.K., Fries, A.(1982)}, \textrm{Fatigue failure models- Birnbaum Saunders versus inverse Gaussian}, \emph{\textrm{IEEE Transactions on Reliability, Vol. 31, pp. 439-440}}.

\bibitem{BS(1969a)}\textrm{Birnbaum, Z.W., Saunders, S.C.(1969a)}, \textrm{A new family of life distribution}, \emph{\textrm{Journal of Applied Probability, Vol. 6, pp. 319-327}}.
\bibitem{BS(1969b)}\textrm{Birnbaum, Z.W., Saunders, S.C.(1969b)}, \textrm{Estimation for a family of life distributions with applications to fatigue}, \emph{\textrm{Journal of Applied Probability, Vol. 6, pp. 328-347}}.

\bibitem{cmb(2015)}\textrm{Chacko,V.M., Mariya Jeeja, P.V., Deepa Paul(2015)}, \textrm{ p- Birnbaum-Saunders distribution:
Applications to reliability and electronic banking habits}, \emph{\textrm{Reliability Theory and Applications, Vol. 10, pp. 70-77}}.


\bibitem{D(1986)} \textrm{Desmond, A.F.(1986)}, \textrm{On the relationship between two fatigue-life models}, \emph{\textrm{IEEE Transactions on Reliability, Vol. 35, pp. 167-169}}.

\bibitem{DRL(2008)} \textrm{Desmond, A.F., Rodriguea-Yam, G.A., Lu,X.(2008)}, \textrm{Estimation of parameters for a Birnbaum-Saunders regression model with censored data}, \emph{\textrm{Journal of Statistical Computation and Simulation, Vol. 78, pp. 983-997}}.

\bibitem{fl(2006)} \textrm{From, S.G., Li, L.(2006)}, \textrm{Estimation of the parameters of the Birnbaum-Saunders distribution}, \emph{\textrm{ Communications in Statistics- Theory and Methods, Vol. 35, pp. 2157-2169}}.

\bibitem{kbj(2010)} \textrm{Kundu, D., Balakrishnan, N., Jamalizadeh, A.(2010)}, \textrm{Bivariate Birnbaum-Saunders distribution and associated inference}, \emph{\textrm{Journal of Multivariate Analysis, Vol. 101, pp. 113-125}}.

\bibitem{kbj(2013)} \textrm{Kundu, D., Balakrishnan, N., Jamalizadeh, A.(2013)}, \textrm{Generalized multivariate Birnbaum-Saunders distributions and related inferential issues}, \emph{\textrm{Journal of Multivariate Analysis, Vol. 116, pp. 230-244}}.

\bibitem{kg(2017)} \textrm{Kundu, D., Gupta, R.C.(2017)}, \textrm{On bivariate Birnbaum-Saunders distribution}, \emph{\textrm{American Journal of Mathematical and  Management Sciences, Vol. 36, pp. 21-33}}.

\bibitem{kkb(2008)} \textrm{Kundu, D., Kannan, N., Balakrishnan, N.(2008)}, \textrm{On the hazard function of Birnbaum-Saunders distribution and associated inference}, \emph{\textrm{Computational Statistics and Data Analysis, Vol. 52, pp. 2692-2702}}.

\bibitem{LC(2009)} \textrm{Lemonte, A. J., Cordeiro, G.M.(2009)}, \textrm{Birnbaum-Saunders nonlinear regression models}, \emph{\textrm{Computational Statistics and Data Analysis, Vol. 53, pp. 4441-4452}}.

\bibitem{nkb(2003)} \textrm{Ng, H.K.T., Kundu, D., Balakrishnan, N.(2003)}, \textrm{Modified  moment estimation for the two-parameter Birnbaum-Saunders distribution}, \emph{\textrm{Computational Statistics and Data Analysis, Vol. 43, pp. 283-298}}.

\bibitem{R(1995)} \textrm{Rieck, J.R.(1995)}, \textrm{Parametric estimation for the Birnbaum-Saunders distribution based on symmetrically censored samples}, \emph{\textrm{ Communications in Statistics- Theory and Methods, Vol. 24, pp. 1721-1736}}.

\bibitem{RN(1991)} \textrm{Rieck, J.R., Nedelman, J.R.(1991)}, \textrm{A log-linear model for the Birnbaum-Saunders distribution}, \emph{\textrm{ Technometrics, Vol. 33, pp. 51-60}}.


\end{thebibliography}

\end{document}